\documentclass[11pt]{amsart}
\usepackage[parfill]{parskip}
\usepackage{fullpage, amssymb, amsmath,mathpazo,hyperref}
\usepackage{graphicx,enumitem,mathabx,color,comment,amscd}
\usepackage{cite}
\newcommand{\p}{\mathfrak p}

\newcommand{\Z}{\mathbf Z}
\newcommand{\A}{\text{A}}
\newcommand{\I}{\text{I}}
\renewcommand{\P}{\text{P}}
\newcommand{\frakP}{\mathfrak{P}}
\newcommand{\Q}{\mathbf Q}
\newcommand{\R}{\mathbf R}

\newcommand{\F}{\mathbf F}
\newcommand{\calS}{\mathcal S}

\renewcommand{\to}{\rightarrow}

\renewcommand{\o}{\mathcal O}
\renewcommand{\phi}{\varphi}
\renewcommand{\epsilon}{\varepsilon}
\renewcommand{\l}{\ell}

\newcommand{\Gal}{\operatorname{Gal}}
\newcommand{\ord}{\text{ord}}
\newcommand{\Jac}{\operatorname{Jac}}
\newcommand{\Mod}[1]{\ (\text{mod}\ #1)}

\newtheorem*{legendre_thm}{Theorem of Legendre}
\newtheorem*{parity}{Parity Conjecture for quadratic twists}
\newtheorem*{hw_thm}{Hasse-Weil bound}
\newtheorem*{chebotarev}{Chebotarev Density Theorem}
\newtheorem*{kronecker-weber}{Kronecker-Weber Theorem}
\newtheorem*{hensel_lem}{Hensel's Lemma}
\newtheorem{thm}{Theorem}[section]
\newtheorem{conj}[thm]{Conjecture}

\newtheorem{lem}[thm]{Lemma}
\newtheorem{prop}[thm]{Proposition}

\theoremstyle{definition}

\newtheorem{ques}[thm]{Question}
\newtheorem{ex}[thm]{Example}
\theoremstyle{remark}
\newtheorem{rem}[thm]{Remark}
\numberwithin{equation}{section}

\title{Squarefree parts of polynomial values}
\author{David Krumm}
\date{}

\begin{document}
\begin{abstract}
Given a separable nonconstant polynomial $f(x)$ with integer coefficients, we consider the set $S$ consisting of the squarefree parts of all the rational values of $f(x)$, and study its behavior modulo primes. Fixing a prime $p$, we determine necessary and sufficient conditions for $S$ to contain an element divisible by $p$. Furthermore, we conjecture that if $p$ is large enough, then $S$ contains infinitely many representatives from every nonzero residue class modulo $p$. The conjecture is proved by elementary means assuming $f(x)$ has degree 1 or 2. If $f(x)$ has degree 3, or if it has degree 4 and has a rational root, the conjecture is shown to follow from the Parity Conjecture for elliptic curves. For polynomials of arbitrary degree, a local analogue of the conjecture is proved using standard results from class field theory, and empirical evidence is given to support the global version of the conjecture. 
\end{abstract}

\maketitle

\section{Introduction}

The \textit{squarefree part} of a nonzero rational number $r$, which we denote here by $S(r)$, is the unique squarefree integer $d$ such that $r/d$ is a square in $\Q$. To any separable polynomial $f(x)\in\Z[x]$ of positive degree we associate the following set of squarefree integers:
\[\calS(f)=\{S(f(r)):r\in\Q\text{\;and\;} f(r)\ne 0\}.\]
An alternate definition of the set $\calS(f)$ can be made in more geometric terms. Let $C$ be the hyperelliptic curve defined by the equation $y^2=f(x)$, and for any squarefree integer $d$, let $C_d$ denote the quadratic twist of $C$ by $d$; i.e., the hyperelliptic curve defined by $dy^2=f(x)$. It may be the case that the curve $C_d$ has rational points at infinity or affine rational points $(x,y)$ with $f(x)=0$; we will refer to these as \textit{trivial} rational points. With this terminology,
\[\calS(f)=\{\text{squarefree}\;d : C_d\; \text{has\;a\;nontrivial\;rational\;point}\}.\]
For any prime number $p$ we consider the problem of determining which residue classes modulo $p$ are represented in the set $\calS(f)$. A simple criterion for deciding whether $\calS(f)$ contains a multiple of $p$ can be readily obtained --- see Theorem \ref{0_thm}. The case of nonzero residue classes is not nearly as simple; in fact, most of this article is devoted to studying that case. Our main goal is to provide evidence in support of the following:
 
\begin{conj}\label{main_conj}
For all but finitely many primes $p$, the set $\calS(f)$ contains infinitely many elements from every nonzero residue class modulo $p$.
\end{conj}

The initial motivation for this work was an observation made in the process of studying quadratic points on the modular curves $Y_1(N)$. Consider, for example, the quadratic number fields $K$ such that the curve $Y_1(18)$ has a $K$-rational point\footnote[2]{Questions about quadratic points on this curve arise in connection to elliptic curves over quadratic fields, and also to dynamical properties of quadratic polynomials over quadratic fields; see \cite[\S 3.11]{jxd}.}. As shown in \cite[Thm. 2.6.5]{krumm_thesis}, if we let
\[f(x)=x^6 + 2x^5 + 5x^4 + 10x^3 + 10x^2 + 4x + 1,\]
then every such field has the form $K=\Q(\sqrt d)$ for some $d\in\calS(f)$. Moreover, every number $d\in\calS(f)$ satisfies $d\equiv 1\Mod 8$ and $d\equiv 0$ or $1\Mod 3$. In view of this result it becomes natural to ask, for any integer $n$, which residue classes modulo $n$ are represented in the set $\calS(f)$. This question can of course be asked not only for this particular polynomial $f(x)$, but for any polynomial. In this article we restrict attention to prime values of $n$ and to separable polynomials in order to simplify various arguments. Thus, we arrive at the problem of determining which residue classes modulo a prime $p$ are represented in a set of the form $\calS(f)$.

Another instance in which this problem becomes of interest is a question concerning congruent numbers. Recall that a positive integer is called \textit{congruent} if it is the area of a right triangle with rational sides. The set of congruent numbers is closely related to the set $\calS(f)$ for the polynomial $f(x)=x^3-x$. Indeed, it is easy to see that a positive integer is congruent if and only if its squarefree part is congruent, and furthermore, it is well known that a positive squarefree integer $d$ is congruent if and only if $d\in\calS(f)$; see \cite[Chap. 1]{koblitz}. The Birch and Swinnerton-Dyer conjecture would imply that $\calS(f)$ contains every positive squarefree integer $d\equiv 5, 6$ or $7\Mod 8$. This result has not been proved unconditionally, but partial results have been achieved. Notably, Monsky \cite{monsky} has shown that $\calS(f)$ contains infinitely many elements from each of the residue classes 5, 6, 7 modulo 8. Now, for any integer $n$ one might ask which residue classes modulo $n$ contain a congruent number. With regards to this question, the following is shown in \S\ref{small_degree_section}.

\begin{thm}\label{congruent_thm} 
For every prime number $p$ there exist infinitely many congruent numbers divisible by $p$. Furthermore, assuming the Parity Conjecture, every residue class modulo $p$ contains infinitely many congruent numbers.
\end{thm}

For polynomials of small degree we can give the best evidence supporting our main conjecture. The next result summarizes our work in \S\ref{small_degree_section}.

\begin{thm}\label{14_thm}
Conjecture \ref{main_conj} holds unconditionally if $f(x)$ has degree 1 or 2.  If $f(x)$ has degree 3, or if it has degree 4 and has a rational root, then the conjecture follows from the Parity Conjecture.
\end{thm}

For polynomials of higher degree we can only show that there is no local obstruction to the conjecture. Letting $C$ denote the hyperelliptic curve $y^2=f(x)$, the conjecture can be rephrased as follows: if $p$ is a large enough prime, then for every integer $m$ coprime to $p$ there exist infinitely many squarefree integers $d$ such that $d\equiv m\Mod p$ and $C_d$ has a nontrivial rational point. Collecting the main results of \S\ref{local_section} we obtain the following local version of the conjecture.

\begin{thm}
Suppose that $f(x)$ has odd degree, or that it has even degree and some irreducible divisor of $f(x)$ has abelian Galois group. Then there is an effectively computable constant $N$ such that the following holds for every prime $p\ge N$: if $m$ is any integer coprime to $p$, then there exist infinitely many squarefree integers $d\equiv m \;(\mathrm{mod}\; p)$  such that the curve $C_d$ has a nontrivial point over every completion of $\Q$.
\end{thm}

\begin{rem}
We expect that the Galois group condition imposed when $f(x)$ has even degree is unnecessary, although our arguments do make essential use of this assumption.
\end{rem}

As an improvement to Conjecture \ref{main_conj} it would be desirable to have a precise quantitative statement about how the elements of $\calS(f)$ are distributed among the residue class modulo $p$. Given a prime $p$ and an integer $m$, we consider the function
\[D(t)=\#\{d\in \calS(f): |d|\le t\;\text{and}\; d\equiv m\Mod p\},\]
and ask what the asymptotic behavior of $D(t)$ is as $t\to\infty$. A similar question has been studied by other authors, but without restricting $d$ to any particular residue class: Stewart-Top \cite[Thm. 2]{stewart-top} show that
 \[\#\{d\in \calS(f): |d|\le t\}\gg t^{1/(g+1)}/(\log t)^2,\]
 where $g$ is the genus of the hyperelliptic curve $y^2=f(x)$, and Granville \cite[Conj. 1.3]{granville} conjectures that if $g\ge 2$, then there is a positive constant $c_f$ such that
 \[\#\{d\in \calS(f): |d|\le t\}\sim c_f\cdot t^{1/(g+1)}.\]
Another useful function to consider is defined as follows. Recall that the \textit{height} of a rational number $r=a/b$ in lowest terms is the number $H(r)=\max\{|a|,|b|\}$. For any real number $t\ge 1$, let
\[\calS(f,t)=\{S(f(r)): f(r)\ne 0 \text{\;and\;} H(r)\le t \}.\]
With $p$ and $m$ as above, we then define
\[E(t)=\#\{d\in \calS(f,t): d\equiv m\Mod p\}\]
and ask how this function behaves as $t\to\infty$. A related question was answered by Poonen in \cite[Thm. 3.5]{poonen}, though again disregarding residue classes: under the assumption of the abc conjecture, if $f(x)$ has degree $\ge 2$, then
\[\#\{S(f(1)), S(f(2)),\ldots, S(f(t))\}\sim t.\] 
If results analogous to those of Stewart-Top or Poonen could be obtained for the functions $D(t)$ or $E(t)$, we would have a proof (perhaps conditional on the abc conjecture) and a more precise version of Conjecture \ref{main_conj}. However, it is not clear whether the methods of the these authors can be modified to yield asymptotics for our counting functions.

This article is organized as follows. In \S\ref{0_class_section} we give necessary and sufficient conditions for the set $\calS(f)$ to contain an element divisible by $p$. A discussion of Conjecture \ref{main_conj} for polynomials of degrees 1-4 is given in \S\ref{small_degree_section}. In \S\ref{local_section} we prove a local analogue of the conjecture for polynomials of arbitrary degree. Finally, \S\ref{data_section} contains empirical evidence supporting Conjecture \ref{main_conj}.

\begin{rem}
Part of the material in this article appears in the author's Ph.D. thesis \cite{krumm_thesis}.
\end{rem}

\section{The residue class of 0}\label{0_class_section}

Let $f(x)\in\Z[x]$ be a nonconstant separable polynomial. We show in this section how to decide, for most primes $p$, whether the residue class of 0 modulo $p$ is represented in the set $\calS(f)$. This particular residue class is easier than others to deal with because there is a simple criterion to decide whether $p$ divides the squarefree part of a nonzero rational number $r$. Indeed, we have $p|S(r)$ if and only if $\ord_p(r)$ is odd. Here, $\ord_p$ denotes the standard $p$-adic valuation.

It will be convenient from this point on to distinguish between two types of primes. Throughout this article we say that an odd prime $p$ is \textit{good} for $f(x)$ if the reduced polynomial $\bar f(x)\in\F_p[x]$ has the same degree as $f(x)$ and has nonzero discriminant. This condition on $p$ implies that the hyperelliptic curve $y^2=f(x)$ has good reduction modulo $p$; see \cite[p. 464, Ex. 1.26]{liu}.

\begin{lem}\label{roots_modq_lem}
Let $R(f)$ be the set of all primes $q$ that are good for $f(x)$ and such that $f(x)$ has a root modulo $q$. Then $R(f)$ is an infinite set.
\end{lem}

\begin{proof}
The Chebotarev Density Theorem implies that there are infinitely many primes $q$ such that $f(x)$ has a root modulo $q$. (See \cite[Thm. 2]{berend-bilu}, for instance.) The result now follows by noting that all but finitely many primes are good for $f(x)$.
\end{proof}

\begin{lem}\label{roots_lem_aux}
 For any finite subset $T$ of $R(f)$ there exists $d\in\calS(f)$ such that $d$ is divisible by every prime in $T$.
\end{lem}

\begin{proof}
Let $t$ be the product of all the primes in $T$, and let
\[A=\{n\in\Z: f(n)\ne 0,\; t|f(n), \;\text{and}\; \gcd(t,f'(n))=1\}.\] 
We claim that $A$ is nonempty. By definition of $R(f)$, for every prime $q\in T$ there is an integer $n_q$ such that $q|f(n_q)$. Moreover, since $q$ is good for $f(x)$, every root of $f(x)$ modulo $q$ must be a simple root; hence, $q$ does not divide $f'(n_q)$. Let $n$ be an integer satisfying $n\equiv n_q\Mod q$ for all $q\in T$, and $f(n)\ne 0$. (The existence of $n$ is guaranteed by the Chinese Remainder Theorem.) For every $q\in T$ we have
\[f(n)\equiv f(n_q)\equiv 0\Mod q \;\;\text{and}\;\;  f'(n)\equiv f'(n_q)\not\equiv 0\Mod q.\]
Therefore $n\in A$, proving that $A$ is nonempty. We partition $A$ into subsets $B$ and $C$ defined by
\[B=\{n\in A: \ord_q(f(n))\equiv 1\Mod 2\;\text{for all}\;q\in T\}\]
and $C=A\setminus B$. Note that if $n\in B$, then the number $d=S(f(n))$ is divisible by every prime in $T$. Hence, the proof will be complete if we show that $B$ is nonempty. Since $A$ is nonempty, at least one of $B$ and $C$ must be nonempty. Assuming $C$ is nonempty, we will show that $B$ must also be nonempty, and this will conclude the proof.

Given $n\in C$, let $W\subseteq T$ consist of all primes $q\in T$ such that $\ord_q(f(n))$ is even. For every $q\in W$, write $\ord_q(f(n))=2s_q$ with $s_q\ge 1$. For every prime $q\in T\setminus W$, set $\ord_q(f(n))=r_q$, which is an odd positive integer. Define $v\in \Z$ by the formula
\[v=\prod_{q\in W}q^{2s_q-1}\cdot\prod_{q\in T\setminus W}q^{r_q+1}.\]
Using a Taylor expansion we see that 
\[f(n+v)=f(n)+f'(n)\cdot v+z\cdot v^2\]
 for some integer $z$. Now, for every $q\in W$ we have 
 \[\ord_q(f(n))=2s_q \;\;\;\text{and}\;\;\;  \ord_q(z\cdot v^2)\ge 2\ord_q(v)=4s_q-2\ge 2s_q.\]
 Moreover, by definition of $A$, $q$ does not divide $f'(n)$, and so
 \[\ord_q(f'(n)\cdot v)=\ord_q(v)=2s_q-1<2s_q.\] 
 Therefore, $\ord_q(f(n+v))=2s_q-1$ is odd. By a similar argument we see that for primes $q\in T\setminus W$, $\ord_q(f(n+v))=r_q$ is also odd. Thus, $f(n+v)$ has odd and positive valuation at every prime in $T$. Finally, for every $q\in T$, $f'(n+v)$ is congruent to $f'(n)$ modulo $q$, and is therefore not divisible by $q$. We conclude that $n+v\in B$, showing that $B$ is nonempty.
\end{proof}

We record a consequence of the proof of Lemma \ref{roots_lem_aux} for use in a later section.

\begin{lem}\label{roots_lem_oneprime}
Let $q$ be a good prime for $f(x)$ such that $f(x)$ has a root modulo $q$. Then there exists an integer $n$ such that $\mathrm{ord}_q(f(n))$ is odd and $q$ does not divide $f'(n)$.
\end{lem}

\begin{proof}
Following the proof of Lemma \ref{roots_lem_aux} with $T=\{q\}$, this result is the statement that the set $B$ is nonempty.
\end{proof}

\begin{prop}\label{roots_prop}
With notation as in Lemma \ref{roots_lem_aux}, there exist infinitely many $d\in\calS(f)$ such that $d$ is divisible by every prime in $T$.
\end{prop}

\begin{proof}
Let $t$ be the product of all the primes in $T$, and fix any prime $q\in R(f)\setminus T$. Applying Lemma \ref{roots_lem_aux} to the finite subset $T\cup\{q\}$ of $R(f)$ we see that there is an element $d_q\in\calS(f)$ that is divisible by $qt$. Since $d_q$ is divisible by $q$, the map $q\mapsto d_q$ is necessarily finite-to-one (as every element of $\calS(f)$ is only divisible by finitely many primes). By Lemma \ref{roots_modq_lem}, the set $R(f)\setminus T$ is infinite, so its image in $\calS(f)$ must also be infinite. Hence, there are infinitely many elements of the form $d_q$ in $\calS(f)$. By construction, every number $d_q$ is divisible by $t$, and hence by every prime in $T$.
\end{proof}

\begin{thm}\label{0_thm} Let $p$ be a good prime for $f(x)$.
\begin{enumerate}
\item If $f(x)$ has odd degree, then there are infinitely many $d\in\calS(f)$ such that $p|d$.
\item Suppose $f(x)$ has even degree. 
\begin{enumerate}
\item If there exists $d\in\calS(f)$ such that $p|d$, then $f(x)$ has a root modulo $p$.
\item Conversely, if $f(x)$ has a root modulo $p$, then there are infinitely many $d\in\calS(f)$ such that $p|d$.
\end{enumerate}
\end{enumerate}
\end{thm}
\begin{proof}
Suppose first that $f(x)$ has odd degree, and write $f(x)=\sum_{i=0}^{2g+1}a_ix^i$. If $g=0$, the result follows by noting that for any prime $q\ne p$ the equation $f(r)=pq$ has a rational solution, so that $pq\in\calS(f)$. As $q$ varies over all primes different from $p$, we obtain infinitely many elements of $\calS(f)$ that are divisible by $p$. Assume now that $g\ge 1$, and let
\[F(x)=x^{2g+2}\cdot f(1/x)=x(a_{2g+1}+a_{2g}x+\cdots+a_0x^{2g+1}).\] 
Since $p$ is good for $f(x)$, $p$ does not divide $a_{2g+1}$. Hence, for every positive integer $n$,
\[\ord_p(F(p^n))=n.\] In particular, if $n$ is odd, then $\ord_p(F(p^n))$ is odd, so $p$ divides the squarefree part of $F(p^n)$. Letting $d=S(F(p^n))$, we then have 
\[p|d\;\;\text{and}\;\; d=S(f(1/p^n))\in\calS(f).\]
Thus, we have shown that the set $D=\{S(F(p^n)): n\text{\;is\;odd}\}$ is contained in $\calS(f)$, and every number in $D$ is divisible by $p$. We claim that $D$ is infinite, which will prove part (1) of the theorem. Note that if $d\in D$, say $d=S(F(p^n))$, then the Diophantine equation $dy^2=F(x)$ has an integral solution (with $x=p^n$). Since $g\ge 1$, Siegel's theorem \cite[Thm. D.9.1]{hindry_silverman} implies that this equation has only a finite number of integral solutions $(x,y)$. Hence, $d$ can only be the squarefree part of $F(p^n)$ for finitely many values of $n$. It follows that the set $D$ is infinite, proving our claim.

To prove part (2), suppose that $f(x)$ has even degree and write $f(x)=\sum_{i=0}^{2g+2}a_ix^i$. Let $F(x,y)\in\Z[x,y]$ be defined by
\[F(x,y)=y^{2g+2}\cdot f(x/y)=a_{2g+2}x^{2g+2}+a_{2g+1}x^{2g+1}y+\cdots+a_0y^{2g+2}.\]
If $\calS(f)$ contains an element divisible by $p$, then there is a rational number $r$ such that $f(r)\ne 0$ and the integer $d=S(f(r))$ is divisible by $p$. Writing $r=a/b$ with $a$ and $b$ coprime integers, we have
\[d=S(f(a/b))=S(F(a,b)).\] 
Hence, there is an integer $s$ such that $F(a,b)=ds^2$. We claim that $p$ cannot divide $b$. For suppose that $p|b$, and reduce the equation $ds^2=F(a,b)$ modulo $p$; we obtain
\[a_{2g+2}\cdot a^{2g+2} \equiv F(a,b) \equiv ds^2\equiv 0 \Mod p.\]
However, this is impossible because $p$ does not divide $a_{2g+2}$ (as $p$ is good for $f(x)$), and $p$ does not divide $a$ (since $p|b$ and $a$ is coprime to $b$). This proves the claim. The equation 
\[b^{2g+2}\cdot f(a/b)=ds^2\] 
therefore takes place in the local ring $\Z_{(p)}$, so we may reduce the equation modulo $p$ to conclude that $\bar f(a/b)=0$, and hence $f(x)$ has a root modulo $p$. This proves part 2(a) of the theorem.

Conversely, suppose that $f(x)$ has a root modulo $p$. Then $p\in R(f)$, so we may apply Proposition \ref{roots_prop} to the set $T=\{p\}\subset R(f)$, and thus obtain that there are infinitely many $d\in\calS(f)$ such that $p|d$. This proves 2(b).
\end{proof}

Having given a simple criterion for deciding whether the residue class of 0 modulo $p$ is represented in $\calS(f)$, we will henceforth restrict attention to nonzero classes.

\section{The case of degrees 1-4}\label{small_degree_section}

For every positive integer $n$ we define a statement $\A(n)$ as follows. 

\textbf{Statement} $\A(n)$. Let $f(x)\in\Z[x]$ be a separable polynomial of degree $n$, and let $p$ be a good prime for $f(x)$. Then for every integer $m$ coprime to $p$ there exist infinitely many elements $d\in\calS(f)$ such that $d\equiv m\Mod p$.

We show in this section that $\A(1)$ and $\A(2)$ hold unconditionally, and that $\A(3)$ is implied by the Parity Conjecture for elliptic curves over $\Q$. Furthermore, still assuming the Parity Conjecture, we show that $\A(4)$ holds for polynomials having a rational root. In order to obtain these results it will be convenient to work with somewhat stronger statements $\I(n)$ and $\P(n)$.

\textbf{Statement} $\I(n)$. Let $h(x)\in\Z[x]$ be a separable primitive\footnote[2]{A polynomial is called \textit{primitive} if the greatest common divisor of its coefficients is equal to 1.} polynomial of degree $n$, and let $p$ be a good prime for $h(x)$. Then for every integer $m$ coprime to $p$ there exist infinitely many primes $q\in\calS(h)$ such that $q\equiv m\Mod p$.

\begin{lem}\label{odd_imp}
For every positive integer $n$, $\mathrm I(n)$ implies $\mathrm A(n)$.
\end{lem}
\begin{proof}
Assume that $\I(n)$ holds. Let $f(x)$ be a separable polynomial of degree $n$, let $p$ be a good prime for $f(x)$, and let $m$ be any integer coprime to $p$. We must show that there are infinitely many elements $d\in\calS(f)$ such that $d\equiv m\Mod p$. By factoring out the greatest common divisor of the coefficients of $f(x)$, we may write $f(x)=\delta s^2\cdot h(x)$ with $\delta$ squarefree and $h(x)$ primitive of degree $n$. Note that since $p$ is good for $f(x)$, it is also good for $h(x)$. Applying  Statement $\I(n)$ to $h(x)$ we see that there exist infinitely many primes $q\in\calS(h)$ such that $q\equiv \delta^{-1}m\Mod p$. Here, $\delta^{-1}$ denotes the multiplicative inverse of $\delta$ modulo $p$. Fix any such prime $q$ that does not divide $\delta$. The integer $d=\delta q$ is then squarefree and congruent to $m$ modulo $p$. Since $q\in\calS(h)$, there are rational numbers $r$ and $t$ such that $h(r)=qt^2$ and $h(r)\ne 0$. Then
\[f(r)=\delta s^2\cdot h(r)= \delta q(st)^2=d(st)^2,\]
so $d=S(f(r))\in\calS(f)$. Since we have infinitely many choices for $q$, this construction yields infinitely many numbers $d\in\calS(f)$ such that $d\equiv m\Mod p$.
\end{proof}

\textbf{Statement} $\P(n)$. Let $h(x)\in\Z[x]$ be a separable polynomial of degree $n$ with square leading coefficient, and let $p$ be a good prime for $h(x)$. Then for every integer $m$ coprime to $p$ there exist infinitely many primes $q\in\calS(h)$ such that $q\equiv m\Mod p$.

\begin{lem}\label{even_imp}
For every positive integer $n$, $\mathrm P(n)$ implies $\mathrm A(n)$.
\end{lem}
\begin{proof}
Assume that $\P(n)$ holds. Let $f(x)$ be a separable polynomial of degree $n$, let $p$ be a good prime for $f(x)$, and let $m$ be any integer coprime to $p$. We must show that there are infinitely many elements $d\in\calS(f)$ such that $d\equiv m\Mod p$. 
Let $\delta$ be the squarefree part of the leading coefficient of $f(x)$, and define $h(x)=\delta\cdot f(x)$. Note that $p$ is good for $h(x)$, and that the leading coefficient of $h(x)$ is a square. Applying  statement $\P(n)$ to $h(x)$ we see that there exist infinitely many primes $q\in\calS(h)$ such that $q\equiv \delta^{-1}m\Mod p$. (Note that $\delta$ has a multiplicative inverse modulo $p$ because $p$ does not divide the leading coefficient of $f(x)$.) Fix any such prime $q$ that does not divide $\delta$. The integer $d=\delta q$ is then squarefree and congruent to $m$ modulo $p$. Since $q\in\calS(h)$, we can write $qs^2=h(r)$ for some rational numbers $r,s$ with $h(r)\ne 0$. Then
\[d(s/\delta)^2=\delta q(s/\delta)^2=h(r)/\delta=f(r),\] so $d=S(f(r))\in\calS(f)$. Since we have infinitely many choices for $q$, this construction yields infinitely many numbers $d\in\calS(f)$ such that $d\equiv m\Mod p$.
\end{proof}

\begin{lem}\label{prime_sys_lem}
Let $a_1,\ldots, a_n$ be pairwise coprime integers, and let  $x_1,\ldots, x_n$ be integers with $\gcd(a_i,x_i)=1$ for all $i$. Then there exist infinitely many prime numbers $q$ such that $q \equiv x_i\;(\mathrm{mod\;} a_i)$ for all $i$.
\end{lem}
\begin{proof}
By the Chinese Remainder Theorem, there is an integer $N$ such that $N\equiv x_i\Mod {a_i}$ for all $i$. Note that $N$ is coprime to each $a_i$, and thus coprime to the number $a=a_1\cdots a_n$. Dirichlet's theorem on primes in arithmetic progressions \cite[p. 251, Thm. 1]{ireland-rosen} then implies that there exist infinitely many primes $q$ that are congruent to $N$ modulo $a$. Clearly, every such prime $q$ satisfies $q\equiv x_i\Mod {a_i}$ for all $i$.
\end{proof}

\begin{prop}\label{degree1_prop}
Statement $\mathrm A(1)$ holds.
\end{prop}
\begin{proof}
By Lemma \ref{odd_imp} it suffices to show that $\I(1)$ holds. Let $f(x)$ be a primitive polynomial of degree 1, so that we can write $f(x)=ax+b$ with $\gcd(a,b)=1$. Let $p$ be a good prime for $f(x)$, and let $m$ be any integer not divisible by $p$. Since $p$ does not divide $a$, Lemma \ref{prime_sys_lem} implies that there exist infinitely many primes $q$ satisfying $q \equiv m \Mod p$ and $q \equiv b\Mod a$. By construction, every such prime has the form $q=b+na=f(n)$ for some $n\in\Z$, and so $q=S(f(n))\in\calS(f)$. This shows that there are infinitely many primes $q\in\calS(f)$ such that $q\equiv m\Mod p$, and $\I(1)$ is proved.
\end{proof}

In order to prove that $\A(2)$ holds we will need the following classical result; see \cite[p. 273]{ireland-rosen} for further details on this theorem.
\begin{legendre_thm}
Let $a,b,c$ be nonzero integers that are squarefree, pairwise coprime, and not all positive nor all negative. Then the equation \[ax^2+by^2+cz^2=0\] has a nontrivial integral solution if and only if the following conditions are satisfied:
\begin{itemize}
\item $-bc$ is a square modulo $a$;
\item $-ac$ is a square modulo $b$;
\item $-ab$ is a square modulo $c$.
\end{itemize} 
\end{legendre_thm}

\begin{prop}\label{degree2_prop}
Statement $\mathrm A(2)$ holds.
\end{prop}
\begin{proof}
By Lemma \ref{even_imp} it suffices to show that $\P(2)$ holds. Let $f(x)$ be a quadratic polynomial with square leading coefficient, so that $f(x)$ has the form $f(x)=a^2x^2+bx+c$ for some integers $a,b,c$. We assume that the discriminant $\Delta=b^2-4a^2c$ is nonzero, and write $\Delta=\delta s^2$ with $\delta$ a squarefree integer and $s\in\Z$. Letting $p$ be a good prime for $f(x)$ and $m$ an integer coprime to $p$, we must show that there exist infinitely many primes $q\in\calS(f)$ such that $q\equiv m\Mod p$.

Let $F(x,y)\in\Z[x,y]$ be the binary quadratic form defined by 
\[F(x,y)=y^2\cdot f(x/y)=a^2x^2+bxy+cy^2.\]
A simple calculation shows that for any prime $q$ we have 
 \[q\in\calS(f) \iff F(x,y) \text{\;represents\;} q \text{\;over\;} \Q.\]
 
 Thus, it suffices to show that there are infinitely many primes $q\equiv m\Mod p$ that are rationally represented by $F(x,y)$. The form $F(x,y)$ is equivalent (over $\Q$) to the diagonal form $G(x,y)=x^2-\delta y^2$. Indeed, letting
 \[
 \begin{pmatrix}
 X\\
 Y
  \end{pmatrix}=
   \begin{pmatrix}
 a & \frac{b}{2a}\\
 0 & \frac{s}{2a}
  \end{pmatrix}
   \begin{pmatrix}
 x\\
 y
  \end{pmatrix}
 \]
 we have $G(X,Y)=F(x,y)$. Since equivalent forms represent the same values, the proof will be complete if we show that there are infinitely many primes $q\equiv m\Mod p$ that are rationally represented by $G(x,y)$. Since $p$ is a good prime for $f(x)$, $p$ is coprime $8\delta$. Hence, by Lemma \ref{prime_sys_lem} there exist infinitely many primes $q$ satisfying
\[ q \equiv m\Mod {p} \;\;\text{and}\;\;q \equiv 1\Mod{8\delta}.\\
\]
Letting $q$ be any such prime, we claim that $q$ is represented by $G(x,y)$ over $\Q$. One can verify using Quadratic Reciprocity that $\delta$ is a square modulo $q$; Legendre's theorem then implies that the equation \[x^2-\delta y^2-qz^2=0\]
has a nontrivial integral solution, say $(x_0,y_0,z_0)$. If $\delta\ne 1$, then we must have $z_0\ne 0$ since $\delta$ is squarefree. In this case we can divide by $z_0^2$ to obtain $G(x_0/z_0,y_0/z_0)=q$. If $\delta=1$, then it is trivial to see that $G(x,y)$ represents $q$: for instance, 
\[G\left(\frac{q+1}{2},\frac{q-1}{2}\right)=q.\] 
This proves our claim and hence the proposition.
\end{proof}

We turn now to consider the statement $\A(3)$. For this statement we do not have an unconditional proof as was the case for $\A(1)$ and $\A(2)$; however, we can bring the machinery of elliptic curves to bear on the problem, and thus provide compelling evidence that $\A(3)$ should hold. Recall that the Parity Conjecture for elliptic curves states that the analytic and algebraic ranks of an elliptic curve over $\Q$ must have the same parity. (See the survey in \cite[Chap. 4]{silverberg} for more details.) We will need to use a different version of this conjecture which relates the rank of an elliptic curve to the rank of a quadratic twist of the curve. For a statement of this conjecture in the literature, see \cite[p. 4]{gouvea-mazur}.

\begin{parity}
Let $E$ be an elliptic curve over $\Q$ with conductor $N_E$, and let $d$ be a squarefree integer coprime to $2\cdot N_E$. Then \[(-1)^{\text{\textnormal{rank}\;} E_d(\Q)}=(-1)^{\text{\textnormal{rank}\;} E(\Q)}\cdot\chi_d(-N_E),\]
where $\chi_d$ is the quadratic Dirichlet character associated to the field $\Q(\sqrt d)$.
\end{parity}

Recall that if $D$ is the discriminant of the field $\Q(\sqrt d)$, then the character
\[\chi_d:(\Z/D\Z)^{\times}\to\{\pm 1\}\] 
can be defined using the Kronecker symbol; 
\[\chi_d(n)=\left(\frac{D}{n}\right).\]
(See \cite[p. 296]{mont-vaughan} for a definition of this symbol.) In particular, $\chi_d$ has the following properties: 

\begin{itemize}
\item $\chi_d(-1)=\text{sign}(d)$
\item If $d\equiv 1\Mod 4$, then $\chi_d(2)=(-1)^{(d^2-1)/8}$.
\item For any odd prime $q$ not dividing $d$, $\chi_d(q)$ is the Legendre symbol $\left(\frac{d}{q}\right)$.
\end{itemize}

\begin{lem}\label{elliptic_lem}
Let $E/\Q$ be an elliptic curve, $p$ an odd prime not dividing the conductor of $E$, and $m$ an integer coprime to $p$. If the Parity Conjecture holds, then there exist infinitely many squarefree integers $d\equiv m \;(\mathrm{mod}\; p)$ such that the twist $E_d$ has positive rank.
\end{lem}

\begin{proof}
Let \[N_E=2^e\cdot\prod_{i=1}^vp_i^{e_i}\] be the prime factorization of the conductor of $E$, and set $\epsilon=(-1)^{1+\text{rank\;} E(\Q)}$. By Lemma \ref{prime_sys_lem}, there exist infinitely many primes $q$ satisfying
\[q\equiv\epsilon\Mod{8p_1\cdots p_v} \text{\;\;and\;\;}q\equiv \epsilon m\Mod p.\]

Fix any such prime $q$. Letting $d=\epsilon\cdot q$, we have
\[d\equiv1\Mod {8p_1\cdots p_v} \text{\;\;and\;\;} d\equiv m\Mod p.\]
Note that $d$ is squarefree and coprime to $2\cdot N_E$. The properties of the character $\chi_d$ imply that
\[\chi_d(-N_E)=\chi_d(-1)\chi_d(2)^e\prod_{i=1}^v\left(\frac{d}{p_i}\right)^{e_i}=\chi_d(-1)=\text{sign}(d)=\epsilon.\]

Hence, by the Parity Conjecture,

\[(-1)^{\text{rank\;} E_d(\Q)}=(-1)^{\text{rank\;} E(\Q)}\cdot\epsilon=(-1)^{\text{rank\;} E(\Q)}(-1)^{1+\text{rank\;} E(\Q)}=-1.\]

It follows that $E_d$ has odd, hence positive, rank. Since we have infinitely many choices for $q$, this argument yields infinitely many squarefree integers $d\equiv m\Mod p$ such that $E_d$ has positive rank.
\end{proof}

\begin{lem}\label{genus1_lem} Let $C/\Q$ be a hyperelliptic curve of genus 1, $p$ an odd prime of good reduction for $C$, and $m$ an integer coprime to $p$. If the Parity Conjecture holds, then there exist infinitely many squarefree integers $d\equiv m \;(\mathrm{mod}\; p)$ such that the Jacobian of $C_d$ has positive rank.
\end{lem}

\begin{proof}
Let $E=\Jac(C)$, which is an elliptic curve over $\Q$. Since $C$ has good reduction modulo $p$, then $E$ also has good reduction modulo $p$. Hence, $p$ does not divide the conductor of $E$; see \cite[p. 256]{silverman}. By Lemma \ref{elliptic_lem}, there exist infinitely many squarefree integers $d\equiv m \;(\mathrm{mod}\; p)$ such that $E_d$ has positive rank. The result now follows by noting that $E_d=\Jac(C_d)$. (This can be deduced from the construction of the Jacobian of a of genus 1 curve; see \cite{mccallum} or \cite[Chap. 20]{cassels}.)
\end{proof}

\begin{prop}\label{degree3_prop}
Statement $\mathrm A(3)$ follows from the Parity Conjecture.
\end{prop}

\begin{proof}
Let $f(x)\in\Z[x]$ be separable of degree 3, let $p$ be a good prime for $f(x)$, and let $m$ be any integer coprime to $p$. We must show, assuming the Parity Conjecture, that there are infinitely many elements $d\in\calS(f)$ such that $d\equiv m\Mod p$. Let $E$ be the elliptic curve over $\Q$ defined by the equation $y^2=f(x)$. Since $p$ is good for $f(x)$, $p$ is an odd prime of good reduction for $E$. Hence, by Lemma \ref{genus1_lem}, the Parity Conjecture implies that there exist infinitely many squarefree integers $d\equiv m\Mod p$ such that the Jacobian of $E_d$ has positive rank. However, for every such $d$, $E_d$ is an elliptic curve and thus isomorphic to its Jacobian; therefore, $E_d$ has positive rank. In particular, $E_d$ must have a nontrivial rational point, and so $d\in\calS(f)$.
\end{proof}

\begin{rem}\label{twist_torsion}
In the proof of Proposition \ref{degree3_prop} we only need the existence of a nontrivial rational point on the twist $E_d$, while the Parity Conjecture (via Lemma \ref{elliptic_lem}) produces a seemingly much stronger result, namely that $E_d$ has positive rank. It is natural to wonder whether one can avoid recourse to the conjecture by only proving the existence of a nontrivial torsion point on $E_d$. Unfortunately, this approach will not work. It follows from a theorem of Silverman \cite[Thm. 6]{silverman_heights} that if $E$ is any elliptic curve over $\Q$, then all but finitely many twists $E_d$ have the property that the torsion subgroup of $E_d(\Q)$ contains only 2-torsion points. Since 2-torsion points on an elliptic curve are -- using our terminology -- trivial rational points, this leaves only a finite number of twists that might have a nontrivial rational torsion point. Hence, we cannot expect by this method to produce the infinitely many integers $d$ required to prove A(3).
\end{rem}

We now discuss our main conjecture in the case of polynomials of degree 4. Though we expect that the statement A(4) holds, it does not seem immediate to deduce it from the Parity Conjecture. In particular, if the proof of Proposition \ref{degree3_prop} is followed with a polynomial $f(x)$ of degree 4 instead of degree 3, the result will be a statement weaker than A(4): instead of obtaining twists $C_d$ with a nontrivial rational point, we would obtain twists $C_d$ whose Jacobian has a nontrivial rational point. There is, nevertheless, one instance in which we can deduce A(4) from the Parity Conjecture:

\begin{prop}\label{degree4_prop}
Assume the Parity Conjecture. Then Statement $\mathrm A(4)$ holds if we restrict attention to polynomials $f(x)$ having a rational root. 
\end{prop}

\begin{proof}
Let $f(x)\in\Z[x]$ be a separable polynomial of degree 4 with a rational root, let $p$ be a good prime for $f(x)$, and let $m$ be any integer coprime to $p$. We must show that there are infinitely many elements $d\in\calS(f)$ such that $d\equiv m\Mod p$. Let $C$ be the hyperelliptic curve of genus 1 defined by the equation $y^2=f(x)$. Note that every quadratic twist of $C$ has a rational point with $y=0$. By Lemma \ref{genus1_lem} there exist infinitely many squarefree integers $d\equiv m\Mod p$ such that the Jacobian of $C_d$ has positive rank. Since $C_d$ has a rational point, it can be given the structure of an elliptic curve, and is therefore isomorphic to its Jacobian. Hence, $C_d$ has infinitely many rational points. In particular, it has a nontrivial rational point, so $d\in\calS(f)$.
\end{proof}

We conclude this section with an application to congruent numbers.

\begin{thm}\label{cong_thm} 
For every prime number $p$ there exist infinitely many congruent numbers divisible by $p$. Furthermore, assuming the Parity Conjecture, every residue class modulo $p$ contains infinitely many congruent numbers.
\end{thm}
\begin{proof}
If $p=2$, this follows (unconditionally) from Monsky's result cited in the introduction. Now fix an odd prime $p$. Recall that a positive squarefree integer $d$ is congruent if and only if $d\in\calS(f)$, where $f(x)=x^3-x$. The discriminant of $f(x)$ is 4, so $p$ is good for $f(x)$.  Since $f(x)$ has odd degree, Theorem \ref{0_thm} implies that there are infinitely many elements of $\calS(f)$ that are divisible by $p$. Note that $f(-x)=-f(x)$, so that the set $\calS(f)$ is closed under taking additive inverses. Hence, we can conclude that there are infinitely many \textit{positive} numbers in $\calS(f)$ divisible by $p$.

To complete the proof we must show, assuming the Parity Conjecture, that every nonzero residue class $[m]$ modulo $p$ contains infinitely many congruent numbers.  The rank of the elliptic curve $E$ defined by the equation $y^2=x^3-x$ is 0, so the proof of Proposition \ref{degree3_prop} can be used to produce infinitely many negative numbers $d\in\calS(f)$ such that $d\equiv -m\Mod p$.  For any such $d$ we have $(-d)\in\calS(f)$, $-d\equiv m\Mod p$, and $-d>0$. This shows that there are infinitely many congruent numbers in $[m]$.
\end{proof}

\section{A local analogue}\label{local_section}

Let $f(x)\in\Z[x]$ be a separable nonconstant polynomial. In the previous section we provided evidence for Conjecture \ref{main_conj} assuming $f(x)$ has degree $\le 4$. For polynomials of higher degree we cannot prove our conjecture or even show that it would be implied by standard conjectures in arithmetic or Diophantine geometry. Hence, we consider instead a weaker statement: our main goal in this section is to show that there is no local obstruction to Conjecture \ref{main_conj}. 

Let $C$ be the hyperelliptic curve $y^2=f(x)$, let $p$ be a prime, and let $m$ be an integer coprime to $p$. Our conjecture states that if $p$ is large enough, then there exist infinitely many squarefree integers $d\equiv m\Mod p$ such that the curve $C_d$ has a nontrivial rational point. We will prove a related statement  in which the global condition that $C_d$ have a nontrivial rational point is replaced by the local condition that it have a nontrivial point over every completion of $\Q$.

For the convenience of the reader, we recall here a few standard results to be used.

\begin{hw_thm}\label{hasse_weil}
Let $X$ be a smooth, irreducible, projective curve of genus $g$ over a finite field $\F_{\l}$. Then \[|\#X(\F_{\l})-(\l+1)|\leq 2g\sqrt \l.\]
\end{hw_thm}

\begin{proof}
See \cite[\S 9.2]{hkt}.
\end{proof}

\begin{hensel_lem}
Let $\l$ be a prime number, and let $P(t_1,\ldots, t_n)\in\Z_{\l}[t_1,\ldots, t_n]$. Suppose $\alpha\in\F_{\l}^n$ is such that $\overline P(\alpha)=0$ and $\nabla \overline P(\alpha)\ne 0$. Then there exists $x\in\Z_{\l}^n$ such that $P(x)=0$ and $\overline x=\alpha$.
\end{hensel_lem}

\begin{proof}
See \cite[p. 15, Cor. 1]{serre}.
\end{proof}

\begin{lem}[Squares in $\l$-adic fields]\label{adic_square_lem}
Let $x\in\Q_{\l}^{\ast}$ and write $x=\l^nu$, where $n\in\Z$ and $u\in\Z_{\l}^{\times}$. Then $x$ is a square in $\Q_{\l}$ if and only if $n$ is even and the following holds:
\begin{itemize}
\item If $\l$ is odd, then $\overline u$ is a square in $\F_{\l}$.
\item If $\l=2$, then $u\equiv 1\;(\mathrm{mod\;} 8)$.
\end{itemize}
\end{lem}

\begin{proof}
See p. 17, Thm. 3 and p. 18, Thm. 4 in \cite{serre}.
\end{proof}

\begin{lem}\label{N0_def} Let $g$ be the genus of $C$. There exists an integer $N_0$ such that every prime $\l\ge N_0$ satisfies
\begin{enumerate}
\item[(a)] $\l$ is good for $f(x)$; and
\item[(b)] for every smooth projective curve $X/\F_{\l}$ of genus $g$, $\#X(\F_{\l})\ge 2g+5$.
\end{enumerate}
\end{lem}
\begin{proof}
For (a) to hold we take $N_0$ to be larger than every prime dividing the discriminant or the leading coefficient of $f(x)$. For (b), the existence of $N_0$ is guaranteed by the Hasse-Weil bound; for instance, it suffices to have $N_0\ge 4g^2+6g+4$.
\end{proof}

\begin{lem}\label{local_lem}
With $N_0$ as in Lemma \ref{N0_def}, let $d$ be a squarefree integer and let $\l\ge N_0$ be prime. 
\begin{enumerate}
\item If $f(x)$ has odd degree, then the curve $C_d$ has a nontrivial point defined over $\Q_{\l}$.
\item Suppose $f(x)$ has even degree. If $\l\nmid d$, or if $\l | d$ and $f(x)$ has a root modulo $\l$, then $C_d$ has a nontrivial point defined over $\Q_{\l}$.
\end{enumerate}
\end{lem}

\begin{proof}
Suppose first that $\l$ does not divide $d$. Since $\l$ is good for $f(x)$, the equation $dy^2=\bar f(x)$ defines a hyperelliptic curve of genus $g$ over $\F_{\l}$, which we denote by $\widetilde{C}_d$. By the definition of $N_0$ we have $\#\widetilde{C}_d(\F_{\l})\ge 2g+5$. The curve $\widetilde{C}_d$ can have at most 2 points at infinity, so it must have at least $2g+3$ affine points defined over $\F_{\l}$. At most $2g+2$ of these points can be of the form $(\alpha,0)$, so there must be a point $(\alpha,\beta)$ with $\beta\ne 0$. Let $F(x,y)=dy^2-f(x)\in\Z_{\l}[x,y]$ and note that
\[\overline F(\alpha,\beta)=0\text{\;\;and\;\;}\nabla \overline F(\alpha,\beta)=\left(-\bar f '(\alpha),2d\beta\right)\ne (0,0)\]
since $\bar f(x)$ has no repeated root. Hensel's Lemma implies that the point $(\alpha,\beta)$ lifts to a point $(a,b)\in\Z_{\l}^2$ with $db^2=f(a)$. Moreover, since $\beta\ne 0$, we must have $b\ne 0$, so $(a,b)$ is a nontrivial point in $C_d(\Q_{\l})$. Thus we have shown that, regardless of what the degree of $f(x)$ is, if $\l$ does not divide $d$, then $C_d$ has a nontrivial point defined over $\Q_{\l}$.

Now suppose that $\l$ divides $d$.

\underline{\textit{Case 1}}: $f(x)$ has odd degree. Let $F(x)=x^{2g+2}\cdot f(1/x)\in\Z[x]$, and let $a$ be the leading coefficient of $f(x)$. A simple calculation shows that we can write $F(ad)=db$, where $b\in\Z $ satisfies $b\equiv a^2\Mod \l$. Thus, $b$ is a nonzero square modulo $\l$. By Lemma \ref{adic_square_lem} there exists a (nonzero) element $y\in\Q_{\l}$ such that $y^2=b$. We have $F(ad)=dy^2$, and so 
\[f(1/ad)=F(ad)/(ad)^{2g+2}=dy^2/(ad)^{2g+2}=d(y/(ad)^{g+1})^2.\]
This shows that the curve $C_d$ has a nontrivial point defined over $\Q_{\l}$.

\underline{\textit{Case 2}}: $f(x)$ has even degree. In this case we assume that $f(x)$ has a root modulo $\l$.  We will show that there is an integer $r$ such that $d\cdot f(r)$ is a nonzero square in $\Q_{\l}$, and this will complete the proof. By Lemma \ref{roots_lem_oneprime} there exists an integer $n$ such that $\ord_{\l}(f(n))$ is odd and $\l\nmid f'(n)$. Write
\[f(n)=\l^{2k-1}s \text{\;\;and\;\;} d=\l t \text{\;\;with\;\;} \l\nmid st.\]
Since $\l$ does not divide $t\cdot f'(n)$, we can choose an integer $b$ such that
\begin{equation}\label{btst}
bt\cdot f'(n)\equiv 1-st\Mod\l.
\end{equation}
Let $a=\l^{2k-1}b$ and $r=n+a$. Using a Taylor expansion we see that
\[f(r)=f(n+a)=f(n)+f'(n)\cdot a+z\cdot a^2\] 
for some integer $z$. We may write $z\cdot a^2=\l^{4k-2}c$ for some $c\in\Z$. Then
\[d\cdot f(r)=d\cdot f(n)+f'(n)(ad)+d(za^2)=\l^{2k}st+f'(n)\l^{2k}bt+\l^{4k-1}ct.\]
Letting $m=st+f'(n)bt+\l^{2k-1}ct$ we have, by \eqref{btst}, 
\[d\cdot f(r)=\l^{2k}\cdot m \text{\;\;and\;\;} m\equiv st+f'(n)bt\equiv 1\Mod{\l}.\]
Lemma \ref{adic_square_lem} then implies that $d\cdot f(r)$ is a square in $\Q_{\l}^{\ast}$.
\end{proof}

We can now prove the local analogue of Conjecture \ref{main_conj} for polynomials of odd degree.

\begin{thm}\label{local_result_odd}
Fix $N_0$ as in Lemma \ref{N0_def}. Suppose that $f(x)$ has odd degree. Let $p\ge N_0$ be prime, and let $m$ be an integer coprime to $p$. There exist infinitely many squarefree integers $d\equiv m \;(\mathrm{mod}\; p)$ such that the curve $C_d$ has a nontrivial point over every completion of $\Q$.
\end{thm}

\begin{proof}
Let $t$ be any integer such that $f(t)\ne 0$, and let $\epsilon\in\{\pm 1\}$ be the sign of $f(t)$. Let $\l_1,\ldots, \l_r$ be all the primes $<N_0$, with $\l_1=2$. We define a squarefree integer $\delta$ by the formula
\[\delta=\epsilon\cdot\prod_{i=1}^r \l_i^{a_i},\]
where $a_i\in\{0,1\}$ is the parity of $\ord_{\l_i}(f(t))$. Note that $\delta\cdot f(t)$ is positive and has even valuation at every prime $\l<N_0$. Hence, we may write $\delta\cdot f(t)=\l^{2e_{\l}}u_{\l}$ with $u_{\l}\in\Z$ not divisible by $\l$. By Lemma \ref{prime_sys_lem} there exist infinitely many primes $q$ such that 
\[q\equiv \delta^{-1}m\Mod p,\;\;  q\equiv u_2^{-1}\Mod 8, \;\;\text{and}\;\; q\equiv u_{\l_i}\Mod {\l_i} \;\;\text{for every $i>1$}.\]
 Fix any such prime $q$, and let $d=q\cdot\delta$. Note that $d$ is squarefree and $d\equiv m\Mod p$. We claim that $d\cdot f(t)$ is a square in every field $\Q_{\l}$ with $\l<N_0$. Indeed, for each odd prime $\l<N_0$ we have
 \[d\cdot f(t)=q\cdot\delta\cdot f(t)=q\cdot\l^{2e_{\l}}\cdot u_{\l}=\l^{2e_{\l}}(q\cdot u_{\l}),\] and by construction $q\cdot u_{\l}$ is a nonzero square modulo $\l$. Hence, $d\cdot f(t)$ is a square in $\Q_{\l}$ by Lemma \ref{adic_square_lem}.  Similarly, we have
 \[d\cdot f(t)=q\cdot\delta\cdot f(t)=q\cdot 2^{2e_{2}}\cdot u_{2}=2^{2e_{2}}(q\cdot u_{2}),\] and by construction $q\cdot u_{2}\equiv 1\Mod 8$. Hence, $d\cdot f(t)$ is a square in $\Q_{2}$. Thus, we have shown that the curve $C_d$ has a nontrivial point defined over $\Q_{\l}$ for every prime $\l<N_0$. By Lemma \ref{local_lem}, the same holds for all primes $\l\ge N_0$. Finally, $C_d$ has a nontrivial point over $\R$ because $d\cdot f(t)$ is positive and hence a square in $\R$. Since we have infinitely many choices for $q$, the above construction yields infinitely many squarefree integers $d\equiv m\Mod p$ such that $C_d$ has a nontrivial point over every completion of $\Q$.
\end{proof}

\begin{rem}\label{odd_deg_rem} We are not aware of  any example of a polynomial $f(x)$ of odd degree and a prime $p$ for which the conclusion of Theorem \ref{local_result_odd} fails; this raises the question of whether the assumption that $p\ge N_0$ is superfluous. For polynomials of even degree, however, a lower bound on $p$ \textit{is} necessary, as will be shown in the examples following Theorem \ref{local_result_even}.
\end{rem}

We turn now to consider the local analogue of Conjecture \ref{main_conj} for polynomials of even degree. Unlike the relatively elementary proof in the case of odd degree, the proof in this case uses results from class field theory. We give here a brief review of the necessary definitions and theorems in order to keep the article self-contained.

To begin we recall some basic facts about Artin symbols, as these will be used in the proof of Lemma \ref{roots_modq}. See \cite[p. 95]{cox} for proofs of these statements. Let $L/K$ be a Galois extension of number fields. Let $\p$ be a prime of $\o_K$ unramified in $\o_L$, and let $\frakP$ be a prime of $\o_L$ lying over $\p$. There is a unique automorphism $\sigma\in\Gal(L/K)$ such that for every $\alpha\in\o_L$,
\[\sigma(\alpha)\equiv \alpha^{N(\p)}\Mod\frakP.\]
Here, $N(\p)$ is the norm of the ideal $\p$. This map $\sigma$ is denoted using the Artin symbol $((L/K)/\frakP)$. As $\frakP$ ranges over all the primes of $\o_L$ containing $\p$, the symbols $((L/K)/\frakP)$ form a  full conjugacy class in the group $\Gal(L/K)$; this class is denoted by $((L/K)/\p)$. We can now state a classical result which is the main tool used in the proof of Lemma \ref{roots_modq}. We refer the reader to \cite[\S 8.B]{cox} and \cite{lagarias-odlyzko} for further details on this theorem.

\begin{chebotarev} Let $C$ be a conjugacy class in the group $\Gal(L/K)$. The set of primes $\p$ of $\o_K$ such that $\p$ is unramified in $\o_L$ and $((L/K)/\p)=C$ has Dirichlet density $(\#C)/[L:K]$.
\end{chebotarev}

In addition to the Chebotarev theorem we will need the following properties of the Artin symbol. These and other basic properties of the symbol are proved in \cite[Thm. 3.9]{lemmermeyer}.
\begin{itemize}
\item The prime $\p$ splits completely in $L$ if and only if $((L/K)/\p)=\{1\}$.
\item Let $F/K$ be a Galois subextension of $L/K$ and let $\mathcal P=\frakP\cap\o_F$. Then
\[\left.\left(\frac{L/K}{\frakP}\right)\right\vert_{F}=\left(\frac{F/K}{\mathcal P}\right).\] 
\end{itemize}

For the remainder of this section we fix an algebraic closure $\overline\Q$ and consider all number fields to be contained in $\overline\Q$. For any positive integer $n$, let $\zeta_n\in\overline\Q$ be a primitive $n$-th root of unity. Recall that for every integer $a$ coprime to $n$ there is an automorphism $\sigma_a\in\Gal(\Q(\zeta_n)/\Q)$ with the property that $\sigma_a(\zeta_n)=\zeta_n^a$. Moreover, the map $[a]\mapsto \sigma_a$ is an isomorphism $(\Z/n\Z)^{\times}\to\Gal(\Q(\zeta_n)/\Q)$. If $p$ is a rational prime that does not divide $n$, then $p$ is unramified in $\Q(\zeta_n)$, so there is a well defined Artin symbol $\left(\frac{\Q(\zeta_n)/\Q}{p}\right)$. One can check that this symbol is in fact equal to $\{\sigma_p\}$.

\begin{lem}\label{roots_modq} 
Let $h(x)\in\Z[x]$ be irreducible and let $L\subset\overline\Q$ be the splitting field of $h(x)$. Let $n$ be a positive integer and let $a$ be an integer coprime to $n$. If the map $\sigma_a$ fixes the field $F=L\cap\Q(\zeta_n)$, then there exist infinitely many primes $p$ such that $p\equiv a\;(\mathrm{mod\;} n)$ and $h(x)$ has a root modulo $p$.
\end{lem}

\begin{proof}
Let $E=L\cdot\Q(\zeta_n)$ be the composite of $L$ and $\Q(\zeta_n)$. By Galois theory \cite[\S 14.4, Prop. 21]{dummit-foote}, the restriction map
\[
\begin{aligned}
\Gal(E/\Q) & \to\Gal(L/\Q)\times\Gal(\Q(\zeta_n)/\Q) \\
\sigma & \mapsto\left(\sigma|_L\;,\; \sigma|_{\Q(\zeta_n)}\right)
\end{aligned}
\]
is an injective group homomorphism with image
\[H = \{(\phi, \tau): \phi|_F=\tau|_F\}.\]
Since $\sigma_a$ fixes $F$, the pair $(1,\sigma_a)$ belongs to $H$. Hence,  there is an automorphism $\sigma\in\Gal(E/\Q)$ such that 
\[\sigma|_L=1 \text{\;\;\;and\;\;\;} \sigma|_{\Q(\zeta_n)}=\sigma_a.\] 
Note that $\sigma$ belongs to the center of $\Gal(E/\Q)$ because the pair $(1,\sigma_a)$ lies in the center of $H$. Therefore, the conjugacy class of $\sigma$ in $\Gal(E/\Q)$ is $\{\sigma\}$. From the Chebotarev Density Theorem it follows that there exist infinitely many rational primes $p$, unramified in $E$, such that $((E/\Q)/p)=\{\sigma\}$. Fix any such prime $p$ that does not divide the leading coefficient of $h(x)$. 

\underline{\textit{Claim 1}}: $p\equiv a\Mod n$. Equivalently, the maps $\sigma_p$ and $\sigma_a$ are equal. Let $\p$ be a prime of $\Q(\zeta_n)$ lying over $p$ and let $\frakP$ be a prime of $E$ lying over $\p$. Then
\[\sigma_p=\left(\frac{\Q(\zeta_n)/\Q}{\p}\right)=\left.\left(\frac{E/\Q}{\frakP}\right)\right\vert_{\Q(\zeta_n)}=\sigma|_{\Q(\zeta_n)}=\sigma_a.\]

\underline{\textit{Claim 2}}: $h(x)$ has a root modulo $p$. Let $\p$ be any prime of $L$ lying over $p$, and $\frakP$ a prime of $E$ lying over $\p$. Then
\[\left(\frac{L/\Q}{\p}\right)=\left.\left(\frac{E/\Q}{\frakP}\right)\right\vert_L=\sigma|_L=1, \;\;\text{so}\;\;\left(\frac{L/\Q}{p}\right)=\{1\}.\]
Therefore, $p$ splits completely in $L$. Let $\theta\in L$ be a root of $h(x)$, and let $c$ be the leading coefficient of $h(x)$. Then the minimal polynomial of $\theta$ over $\Q$ is $\frac{1}{c}\cdot h(x)$, which has coefficients in the local ring $\Z_{(p)}$, so $\theta$ is integral over $\Z_{(p)}$. Hence, for any prime $\p$ of $L$ lying over $p$, $\theta$ belongs to the local ring $\mathcal O_{L,\p}$. (This can be seen using \cite[Prop. 5.12]{atiyah-macdonald}.) Since $p$ splits completely in $L$, the residue field of $\mathcal O_{L,\p}$ is $\Z/p\Z$. Thus, reducing the equation $h(\theta)=0$ modulo $\p$ we obtain a relation $\overline h(\overline\theta)=0$ over $\Z/p\Z$, and this proves the claim.
\end{proof}

We can now prove our main theorem for polynomials of even degree by combining Lemma \ref{roots_modq} with the following well known result.
\begin{kronecker-weber}
For every abelian extension $L$ of $\Q$ there is a positive integer $b$ such that $L\subseteq\Q(\zeta_b)$. Moreover, one can take $b$ to be the absolute value of the discriminant of $L$.
\end{kronecker-weber}
\begin{proof}
See \cite[\S V.1, Thm. 1.10]{neukirch} and \cite[Thm. 3.3]{lemmermeyer}.
\end{proof}

\begin{thm}\label{local_result_even}
Suppose that $f(x)$ has even degree and that some irreducible divisor of $f(x)$ has abelian Galois group. Then there is a constant $N$ such that the following holds for every prime $p\ge N$: if $m$ is any integer coprime to $p$, then there exist infinitely many squarefree integers $d\equiv m \;(\mathrm{mod}\; p)$  such that the curve $C_d$ has a nontrivial point over every completion of $\Q$.
\end{thm}

\begin{proof} 
Let $t$ be any integer such that $f(t)\ne 0$ and let $\delta$ be the squarefree part of $f(t)$. Let $h(x)\in\Z[x]$ be an irreducible divisor of $f(x)$ such that the splitting field $L\subset\overline\Q$ of $h(x)$ has abelian Galois group. By the Kronecker-Weber theorem, there is an effectively computable positive integer $b$ such that $L\subseteq\Q(\zeta_b)$. Fix an integer $N_0$ as in Lemma \ref{N0_def}, and let $N$ be any integer satisfying
\[N>\max\{b, N_0, |f(t)|\}.\] 
We claim that $N$ has the property given in the statement of the theorem. Let $p\ge N$ be prime and let $m$ be an integer coprime to $p$. We will show that there are infinitely many squarefree integers $d\equiv m\Mod p$  such that  $C_d$ has a nontrivial point over every completion of $\Q$. Let $2, \l_1,\ldots, \l_r$ be all the primes $<N$, and set 
\[\alpha=8p\cdot\l_1\cdots \l_r\;\;,\;\;n=\alpha\cdot b.\]
Since $\zeta_n^{\alpha}$ is a primitive $b$-th root of unity, we have $L\subseteq\Q(\zeta_n^{\alpha})\subseteq\Q(\zeta_n)$. Note that $p$ does not divide $8b\cdot\l_1\cdots\l_r$ since, by construction, $p\ge N>\max\{2,\l_1,\ldots, \l_r, b\}$. Similarly, $p$ does not divide $\delta$ since $p\ge N>|f(t)|\ge|\delta|$. Hence, there exists an integer $a$ satisfying
\[a\equiv 1\Mod{8b\cdot\l_1\cdots\l_r} \;\;\text{and}\;\; a\equiv \delta^{-1}m\Mod p.\]
From the definitions it follows that $a$ is coprime to $n$. We claim that the map $\sigma_a\in\Gal(\Q(\zeta_n)/\Q)$ fixes $L$; in fact, it fixes the larger field $\Q(\zeta_n^{\alpha})$. Indeed, we have
\[\sigma_a(\zeta_n^{\alpha})=\sigma_a(\zeta_n)^{\alpha}=\zeta_n^{a\alpha}=\zeta_n^{\alpha}\]
because $n|\alpha(a-1)$. By Lemma \ref{roots_modq} there exist infinitely many primes $q$ such that $q\equiv a\Mod n$ and $h(x)$ has a root modulo $q$. In particular, $f(x)$ has a root modulo $q$. Fix any such prime $q\ge N$ and let $d=\delta\cdot q$. Note that $d$ is squarefree and $d\equiv m\Mod p$. By Lemma \ref{local_lem}, the curve $C_d$ has a nontrivial point over every field $\Q_{\l}$ with $\l\ge N$. (The only prime $\l\ge N$ that divides $d$ is $\l=q$, and $f(x)$ has a root modulo $q$.) Moreover, since $d\cdot f(t)$ is positive, $C_d$ also has a nontrivial point over $\R$. By construction, $q\equiv a\equiv 1\Mod{8\cdot\l_1\cdots\l_r}$, so Lemma \ref{adic_square_lem} implies that $q$ is a square in $\Q_2, \Q_{\l_1}, \ldots, \Q_{\l_r}$. Now, $d\cdot f(t)=q\cdot\delta\cdot f(t)$, and $\delta\cdot f(t)$ is a square integer, so $d\cdot f(t)$ is a square in $\Q_{\l}$ for all primes $\l<N$. Hence, $C_d$ has a nontrivial point over every field $\Q_{\l}$ with $\l<N$. Thus, we have shown that $C_d$ has a nontrivial point over every completion of $\Q$. Since we have infinitely many choices for $q$, this construction provides infinitely many squarefree integers $d\equiv m\Mod p$  such that  $C_d$ has a nontrivial point over every completion of $\Q$.
\end{proof}

\begin{rem}
All currently available evidence suggests that the Galois group condition included in the hypotheses of Theorem \ref{local_result_even} is not necessary. However, a lower bound on $p$ \textit{is} necessary -- otherwise the conclusion of the theorem may not hold. The examples below will illustrate some of the issues that can arise at small primes.
\end{rem}

\begin{ex}\label{bad_prime_ex}
Let $f(x)=x^6 + 2x^5 + 5x^4 + 10x^3 + 10x^2 + 4x + 1$. The prime $p=3$ is not good for $f(x)$, because $f(x)$ has a multiple root modulo 3 (namely, $x=1$). We claim that if $d\equiv 2\Mod 3$ is any squarefree integer, then the curve $C_d$ does not have a nontrivial point over $\Q_3$. In particular, this implies that there does not exist a squarefree integer $d\equiv 2\Mod 3$ such that $C_d$ has a nontrivial point over every completion of $\Q$. Hence, the conclusion of Theorem \ref{local_result_even} does not hold if $p=3, m=2$. To prove the claim, suppose that $d$ is a squarefree integer such that $C_d$ has a nontrivial point over $\Q_3$. We will show that $d\equiv 0, 1\Mod 3$. By hypothesis, there exist $a, b\in\Z_3$ and $s\in\Q_3^{\ast}$ such that $ds^2=f(a/b)$. We may assume without loss of generality that $a$ and $b$ are not both divisible by 3. Define a polynomial $F(x,y)\in\Z[x,y]$ by
\[F(x,y)=y^6\cdot f(x/y)=x^6 + 2x^5y + 5x^4y^2 + 10x^3y^3 + 10x^2y^4 + 4xy^5 + y^6.\] 
Letting $t=sb^3$, we have
\[F(a,b)=b^6\cdot f(a/b)=b^6\cdot ds^2=dt^2.\]
Considering the 3-adic valuation of both sides of this equation we conclude that $t\in\Z_3$. Now, allowing $a$ and $b$ to take all possible values  modulo 9 we find that  if $F(a,b)\equiv 0\Mod 9$, then $a, b\equiv 0\Mod 3$, which is a contradiction. Hence, $F(a,b)$ is not divisible by 9, so $t$ is not divisible by 3, and therefore  $t^2\equiv 1\Mod 3$. Thus,
 \[d\equiv dt^2\equiv F(a,b)\Mod 3.\] The conclusion that $d\equiv 0, 1\Mod 3$ now follows from the easily verified fact that the map $F:\F_3\times\F_3\to\F_3$ only takes the values 0 and 1.
\end{ex}

The above example shows that the conclusion of Theorem \ref{local_result_even} can fail if $p$ is not good for $f(x)$. The next example will show that the same can happen if $p$ is a good prime that is too small.

\begin{lem}\label{no_pt_lem} 
Suppose that $f(x)$ has even degree. Let $p$ be a good prime for $f(x)$, and assume that $C(\F_p)=\emptyset$. Let $d$ be any squarefree integer that is a nonzero square modulo $p$. Then $C_d$ does not have a nontrivial point over $\Q_p$.
\end{lem}
\begin{proof}
Suppose that $C_d$ does have a nontrivial point over $\Q_p$, so that there exist $a, b\in\Z_p$ and $s\in\Q_p^{\ast}$ such that $ds^2=f(a/b)$. We may assume without loss of generality that $a$ and $b$ are not both divisible by $p$. We consider two cases, depending on whether $b$ is divisible by $p$.

Suppose first that $b\not\equiv 0\Mod p$, so that $a/b\in\Z_p$. Considering the $p$-adic valuation of both sides of the equation $ds^2=f(a/b)$ we see that necessarily $s\in\Z_p$. This equation then takes place in $\Z_p$, so we may reduce modulo $p$ to obtain a solution to the equation $\bar dy^2=f(x)$ with $x, y\in\F_p$. Since $d$ is a square modulo $p$, we may write $\bar d=\alpha^2$ for some $\alpha\in\F_p$. Then we have a point $(x,\alpha y)\in C(\F_p)$, contradicting the assumption that $C(\F_p)=\emptyset$.

Suppose now that $b\equiv 0\Mod p$. Write $f(x)=c_{2k}x^{2k}+\cdots +c_1x+ c_0$ with $c_{2k}\ne 0$, and let
\[F(x,y)=y^{2k}\cdot f(x/y)=c_{2k}x^{2k}+c_{2k-1}x^{2k-1}y+\cdots+c_1xy^{2k-1}+c_0y^{2k}.\]
We have \[F(a,b)=b^{2k}\cdot f(a/b)=b^{2k}ds^2=dt^2,\]
 where $t=sb^k$. Considering $p$-adic valuations we see that $t\in\Z_p$. The above equation can then be reduced modulo $p$ to obtain $c_{2k}a^{2k}\equiv dt^2\Mod p$. In particular, $c_{2k}a^{2k}$ is a square modulo $p$. Since $p$ does not divide $a$ (because it divides $b$), this implies that $c_{2k}$ is a square modulo $p$; thus, $C(\F_p)$ contains two points at infinity.  Once again, this contradicts the assumption that $C(\F_p)=\emptyset$. Since both cases have led to a contradiction, we conclude that $C_d$ cannot have a nontrivial point over $\Q_p$.
\end{proof}

\begin{ex}\label{small_good_prime_ex}
Let $f(x)=2x^8-x^6-8x^4-x^2+2$. The prime $p=19$ is good for $f(x)$, and the curve $C: y^2=f(x)$ has no point over the field $\F_{19}$. (This and other examples of ``pointless" curves over finite fields can be found in the articles \cite{howe_pointless, maisner_nart}.) It follows from Lemma \ref{no_pt_lem} that if $m$ is a quadratic residue modulo 19, then there does not exist a squarefree integer $d\equiv m\Mod {19}$ such that $C_d$ has a nontrivial point over $\Q_{19}$. Hence, the conclusion of Theorem \ref{local_result_even} does not hold with these values of $p$ and $m$.
\end{ex}

\section{Empirical data}\label{data_section}

In this last section we summarize the results of two numerical experiments designed to test Conjecture \ref{main_conj}. Unless otherwise specified, all computations were done using Sage \cite{sage}.

The first experiment is a partial verification of the conjecture for a systematically chosen collection of polynomials. Given a separable nonconstant polynomial $f(x)\in\Z[x]$, the conjecture states that for every large enough prime $p$, the set $\calS(f)$ contains infinitely many elements from every nonzero residue class modulo $p$. For purposes of checking this statement in practice, the conjecture must be modified so that it can be reduced to a finite computation. In particular, only a finite number of primes $p$ can be considered, and only a finite subset of $\calS(f)$ can be computed. Hence, in order to test the conjecture we take the following steps:
\begin{enumerate}
\item Choose a finite set of primes, say, all primes $p\le n$ for a fixed positive integer $n$.
\item Compute a finite subset of $\calS(f)$: fix a height bound $B$ and let
\[\widetilde\calS(f)=\{S(f(r)):r\in\Q, f(r)\ne 0,\text{\;and\;} H(r)\le B\}.\]
\item For every prime $p\le n$, check whether $\widetilde\calS(f)$ contains an element from every nonzero residue class modulo $p$.
\end{enumerate}

If $p$ is a prime for which the check in step (3) fails, we will say that $p$ is an \textit{exceptional prime} for $f(x)$ (although this term depends on the choice of height bound $B$). Conjecture \ref{main_conj} does not preclude the possibility that $f(x)$ may have exceptional primes, but it does suggest that all such primes must be relatively small. Thus, if the conjecture holds, we should not expect to find any large exceptional primes when following the steps above. For our first experiment we chose a large collection of polynomials and carried out a computation which shows that, as expected, all exceptional primes for the selected polynomials are quite small.

\begin{prop}\label{data_prop} Let $f(x)\in\Z[x]$ be a separable polynomial of degree $4,5, 6,7,8$ or 9, all of whose coefficients have absolute value at most 3. Then for every prime $19< p< 10^3$, the set $\calS(f)$ contains a representative from every nonzero residue class modulo $p$.
\end{prop}

\begin{proof}
For every polynomial $f(x)$ satisfying the conditions of the proposition, we follow the three steps listed above taking $n=10^3$ and $B=400$. The largest exceptional prime that occurs in the entire computation is 19. Hence, for every prime $19< p< 10^3$, the set $\widetilde\calS(f)$ contains a representative from every nonzero residue class modulo $p$. Since $\calS(f)\supseteq\widetilde\calS(f)$, the same holds for $\calS(f)$.
\end{proof}

\begin{rem} The total number of polynomials satisfying the conditions of Proposition \ref{data_prop} is 17,896.
\end{rem}

For the second experiment we select only four polynomials, but carry out a more extensive computation in terms of the number of primes considered. For each of the selected polynomials $f(x)$ we set a height bound of $B=800$ and compute the set $\widetilde\calS(f)$ as defined earlier. This set is then reduced modulo $p$ for every prime $p<10,000$, and a list of exceptional primes is kept. Once again, all exceptional primes that were found are very small. The precise results of our computation are as follows:

\begin{itemize}
\item For the fifth cyclotomic polynomial, $f(x)=x^4+x^3+x^2+x+1$, the only exceptional prime below $10,000$ is $p=5$. The reduction of the set $\widetilde\calS(f)$ modulo 5 is $\{0, 1\}$.

\item For the polynomial $f(x)=x^7-3$ there is no exceptional prime below 10,000. This example illustrates a more general fact: we have not found \textit{any} polynomial $f(x)$ of odd degree having an exceptional prime.

\item For the polynomial from Example \ref{bad_prime_ex},  namely $f(x)=x^6 + 2x^5 + 5x^4 + 10x^3 + 10x^2 + 4x + 1$, the only exceptional prime below $10,000$ is $p=3$. The reduction of the set $\widetilde\calS(f)$ modulo 3 is $\{0, 1\}$.

\item For the polynomial $f(x)=2x^8-x^6-8x^4-x^2+2$, which appears in Example \ref{small_good_prime_ex}, the primes 2, 3, and 97 are not good, and turn out to be exceptional primes; the only good exceptional prime below $10,000$ is $p=19$. The reduction of the set $\widetilde\calS(f)$ modulo 19 consists of all non-squares modulo 19, a result that is consistent with Example \ref{small_good_prime_ex}.
\end{itemize}

Using Magma \cite{magma} we determine that the latter two polynomials are irreducible and have nonabelian Galois group. Hence, the results of this experiment support the statement that the conclusion of Theorem \ref{local_result_even} remains valid without the assumption that some divisor of $f(x)$ has abelian Galois group.

We end this article by posing two questions motivated by the above experiments.

\begin{ques}
Does the conclusion of Theorem \ref{local_result_even} continue to be true if the hypothesis of an irreducible divisor with abelian Galois group is removed?
\end{ques}

\begin{ques}
For polynomials of odd degree, can Conjecture \ref{main_conj} be strengthened to include all primes, rather than all but finitely many?
\end{ques}

\subsection*{Acknowledgements} I thank Dino Lorenzini for helpful conversations during the preparation of this paper, and for valuable comments on an earlier draft.

\bibliography{ref_list}{}
\bibliographystyle{amsplain}

\end{document}